\documentclass[12pt,a4paper]{article}
\usepackage{indentfirst}
\usepackage{amsmath,amssymb,amsfonts,amsthm,graphics}
\usepackage{makeidx}
\usepackage{color}

\newtheorem{theorem}{Theorem}
\newtheorem{definition} {Definition}

\newtheorem{lemma}{Lemma}

\newtheorem{example}{Example}
\newtheorem{corollary}{Corollary}

\newtheorem{remark}{Remark}
\newtheorem{observation}{Observation}

\begin{document}
\title{\Large\bf Oriented diameter and rainbow\\ connection number of a graph\footnote{Supported by NSFC
No.11071130, and  ``the Fundamental Research Funds for the Central
Universities".}}
\author{\small Xiaolong Huang, Hengzhe Li, Xueliang Li, Yuefang Sun
\\
\small Center for Combinatorics and LPMC-TJKLC
\\
\small Nankai University, Tianjin 300071, China
\\
\small huangxiaolong@mail.nankai.edu.cn;
lhz2010@mail.nankai.edu.cn;\\\small lxl@nankai.edu.cn;
bruceseun@gmail.com}
\date{}
\maketitle
\begin{abstract} The oriented diameter of
a bridgeless graph $G$ is $\min\{ diam(H)\ | H\ is\ an\\\
orientation\ of\ G\}$. A path in an edge-colored graph $G$, where
adjacent edges may have the same color, is called rainbow if no two
edges of the path are colored the same. The rainbow connection
number $rc(G)$ of $G$ is the smallest integer $k$ for which there
exists a $k$-edge-coloring of $G$ such that every two distinct
vertices of $G$ are connected by a rainbow path. In this paper, we
obtain upper bounds for the oriented diameter and the rainbow
connection number of a graph in terms of $rad(G)$ and $\eta(G)$,
where $rad(G)$ is the radius of $G$ and $\eta(G)$ is the smallest
integer number such that every edge of $G$ is contained in a cycle
of length at most $\eta(G)$. We also obtain constant bounds of the
oriented diameter and the rainbow connection number for a
(bipartite) graph $G$ in terms of the minimum degree of $G$.

\end{abstract}

\section{Introduction}

All graphs in this paper are undirected, finite and simple. We refer
to book \cite{bondy} for notation and terminology not described
here. A path $u=u_1,u_2,\ldots,u_k=v$ is called a $P_{u,v}$ path.
Denote by $u_iPu_j$ the subpath $u_i,u_{i+1},\ldots,u_j$ for $i\leq
j$. The $length\ \ell(P)$ of a path $P$ is the number of edges in
$P$. The $distance$ between two vertices $x$ and $y$ in $G$, denoted
by $d_{G}(x,y)$, is the length of a shortest path between them. The
$eccentricity$ of a vertex $x$ in $G$ is $ecc_{G}(x)=max_{y\in
V(G)}d(x,y)$. The $radius$ and $diameter$ of $G$ are
$rad(G)=min_{x\in V(G)}ecc(x)$ and $diam(G)=max_{x\in V(G)}ecc(x)$,
respectively. A vertex $u$ is a $center$ of a graph $G$ if
$ecc(u)=rad(G)$. The oriented diameter of a bridgeless graph $G$ is
$\min\{\,diam(H)\ |$ $H\ is\ an\ orientation\ of\ G\}$, and the
oriented radius of a bridgeless graph $G$ is $\min\{\,rad(H)\ |\, H\
is\ an\ orientation\ of\ G\}$. For any graph $G$ with
edge-connectivity $\lambda(G)=0,1$, $G$ has oriented radius (resp.
diameter) $\infty$.

In 1939, Robbins solved the One-Way Street Problem and proved that a
graph $G$ admits a strongly connected orientation if and only if $G$
is bridgeless, that is, $G$ does not have any cut-edge. Naturally,
one hopes that the oriented diameter of a bridgeless graph is as
small as possible. Bondy and Murty suggested to study the
quantitative variations on Robbins' theorem. In particular, they
conjectured that there exists a function $f$ such that every
bridgeless graph with diameter $d$ admits an orientation of diameter
at most $f(d)$.

In 1978, Chv\'atal and Thomassen \cite{chv} obtained some general
bounds.

\begin{theorem}[\bf Chv\'atal and Thomassen 1978 \cite{chv}]\label{T1}
For every bridgeless graph $G$, there exists an orientation $H$ of
$G$ such that
\[rad(H)\leq rad(G)^2+rad(G),\]
\[diam(H)\leq 2rad(G)^2+2rad(G).\] Moreover, the above bounds are
optimal.
\end{theorem}

There exists a minor error when they constructed the graph $G_d$
which arrives at the upper bound when $d$ is odd. Kwok, Liu and West
gave a slight correction in \cite{kwo}.

They also showed that determining whether an arbitrary graph can be
oriented so that its diameter is at most 2 is NP-complete. Bounds
for the oriented diameter of graphs have also been studied in terms
of other parameters, for example, radius, dominating number
\cite{chv,fom,kwo,sol}, etc. Some classes of graphs have also been
studied in \cite{fom, koh1,koh2,kon,mcc}.

Let $\eta(G)$ be the smallest integer such that every edge of $G$
belongs to a cycle of length at most $\eta(G)$. In this paper, we
show the following result.
\begin{theorem}\label{T2}
For every bridgeless graph $G$, there exists an orientation $H$ of
$G$ such that
\[rad(H)\leq \sum_{i=1}^{rad(G)}\min\{2i,\eta(G)-1\}\leq
rad(G)(\eta(G)-1),\]
\[diam(H)\leq 2\sum_{i=1}^{rad(G)}\min\{2i,\eta(G)-1\}\leq
2rad(G)(\eta(G)-1).\]
\end{theorem}

Note that $\sum_{i=1}^{rad(G)}\min\{2i,\eta(G)-1\}\leq
rad(G)^2+rad(G)$ and $diam(H)\leq 2rad(G)$. So our result implies
Chv\'atal and Thomassen's Theorem \ref{T1}.

A path in an edge-colored graph $G$, where adjacent edges may have
the same color, is called $rainbow$ if no two edges of the path are
colored the same. An edge-coloring of a graph $G$ is a $rainbow\
edge$-$coloring$ if every two distinct vertices of graph $G$ are
connected by a rainbow path. The $rainbow\ connection\ number\
rc(G)$ of $G$ is the minimum integer $k$ for which there exists a
rainbow $k$-edge-coloring of $G$. It is easy to see that
$diam(G)\leq rc(G)$ for any connected graph $G$. The rainbow
connection number was introduced by Chartrand et al. in \cite{char}.
It is of great use in transferring information of high security in
multicomputer networks. We refer the readers to \cite{chak} for
details.

Chakraborty et al. \cite{chak} investigated the hardness and
algorithms for the rainbow connection number, and showed that given
a graph $G$, deciding if $rc(G)=2$ is $NP$-complete. Bounds for the
rainbow connection number of a graph have also been studies in terms
of other graph parameters, for example, radius, dominating number,
minimum degree, connectivity, etc. \cite{bas,char,kri}. Cayley
graphs and line graphs were studied in \cite{liliu} and \cite{lis},
respectively.

A subgraph $H$ of a graph $G$ is called $isometric$ if the distance
between any two distinct vertices in $H$ is the same as their
distance in $G$. The size of a largest isometric cycle in $G$ is
denoted by $\zeta(G)$. Clearly, every isometric cycle is an induced
cycle and thus $\zeta(G)$ is not larger than the chordality, where
$chordality$ is the length of a largest induced cycle in $G$. In
\cite{bas}, Basavaraju, Chandran, Rajendraprasad and Ramaswamy got
the the following sharp upper bound for the rainbow connection
number of a bridgeless graph $G$ in terms of $rad(G)$ and
$\zeta(G)$.

\begin{theorem}[\bf Basavaraju et al. \upshape\cite{bas}]\label{T3}
For every bridgeless graph $G$,
\[rc(G)\leq\sum_{i=1}^{rad(G)}\min\{2i+1,\zeta(G)\}\leq rad(G)\zeta(G).\]
\end{theorem}

In this paper, we show the following result.
\begin{theorem} \label{T4}
For every bridgeless graph $G$,
\[rc(G)\leq \sum_{i=1}^{rad(G)}\min\{2i+1,\eta(G)\}\leq
rad(G)\eta(G).\]
\end{theorem}

From Lemma~\ref{L2} of Section~2, we will see that $\eta(G)\leq
\zeta(G)$. Thus our result implies Theorem~\ref{T3}.

This paper is organized as follows: in Section $2$, we introduce
some new definitions and show several lemmas. In Section $3$, we
prove Theorem~\ref{T2} and study upper for the oriented radius
(resp. diameter) of plane graphs, edge-transitive graphs and general
(bipartite) graphs. In Section $4$, we prove Theorem~\ref{T4} and
study upper for the rainbow connection number of plane graphs,
edge-transitive graphs and general (bipartite) graphs.

\section{Preliminaries}

In this section, we introduce some definitions and show several
lemmas.
\begin{definition}\upshape For any $x\in V(G)$ and $k\geq 0$,
the $k$-$step\ open\ neighborhood$ is $\{y\,|\,d(x,y)=k\}$ and
denoted by $N_k(x)$, the $k$-$step\ closed\ neighborhood$ is
$\{y\,|\,d(x,y)\linebreak[3]\leq k\}$ and denoted by $N_k[x]$. If
$k=1$, we simply write $N(x)$ and $N[x]$ for $N_1(x)$ and $N_1[x]$,
respectively.\end{definition}

\begin{definition}\upshape
Let $G$ be a graph and $H$ be a subset of $V(G)$ (or a subgraph of
$G$). The edges between $H$ and $G\setminus H$ are called $legs$ of
$H$. An $H$-$ear$ is a path $P=(u_0,u_1,\ldots,u_k)$ in $G$ such
that $V(H)\cap V(P)=\{u_0,u_k\}$. The vertices $u_0,\,u_k$ are
called the $foot$ of $P$ in $H$ and $u_0u_1,\,u_{k-1}u_k$ are called
the $legs$ of $P$. The $length$ of an $H$-ear is the length of the
corresponding path. If $u_0=u_k$, then $P$ is called a $closed\
H$-$ear$. For any leg $e$ of $H$, denote by $\ell(e)$ the smallest
number such that there exists an $H$-ear of length $\ell(e)$
containing $e$, and such an $H$-ear is called an $optimal$
$(H,e)$-ear.
\end{definition}

Note that for any optimal $(H,e)$-ear $P$ and every pair $(x,y)\neq
(u_0,u_k)$ of distinct vertices of $P$, $x$ and $y$ are adjacent on
$P$ if and only if $x$ and $y$ are adjacent in $G$.

\begin{definition}\upshape
For any two paths $P$ and $Q$, the joint of $P$ and $Q$ are the
common vertex and edge of $P$ and $Q$. Paths $P$ and $Q$ have $k\
continuous\ common\ segments$ if the common vertex and edge are $k$
disjoint paths.
\end{definition}

\begin{definition}\upshape
Let $P$ and $Q$ be two paths in $G$. Call $P$ and $Q$ $independent$
if they has no common internal vertex.
\end{definition}

\begin{lemma}\label{L1}
Let $n\geq 1$ be an integer, and let $G$ be a graph, $H$ be a
subgraph of $G$ and $e_i=u_iv_i$ be a leg of $H$ and
$P_i=P_{u_iw_i}$ be an optimal $(G,e_i)$-ear, where $1\leq i\leq n$
and $u_i,w_i$ are the foot of $P_i$. Then for any leg
$e_j=u_jv_j\neq e_i,\,1\leq i\leq n$, either there exists an optimal
$(H,e_j)$-ear $P_j=P_{u_jw_j}$ such that either $P_i$ and $P_j$ are
independent for any $P_i,\,1\leq i\leq n$, or $P_i$ and $P_j$ have
only one continuous common segment containing $w_j$ for some $P_i$.
\end{lemma}
\begin{proof} Let $P_j$ be an optimal $(H,e_j)$-ear. If $P_i$ and $P_j$
are independent for any $i$, then we are done. Suppose that $P_i$
and $P_j$ have $m$ continuous common segments for some $i$, where
$m\geq 1$.
\begin{figure}[h,t,b]
\begin{center}
\scalebox{0.7}[0.7]{\includegraphics{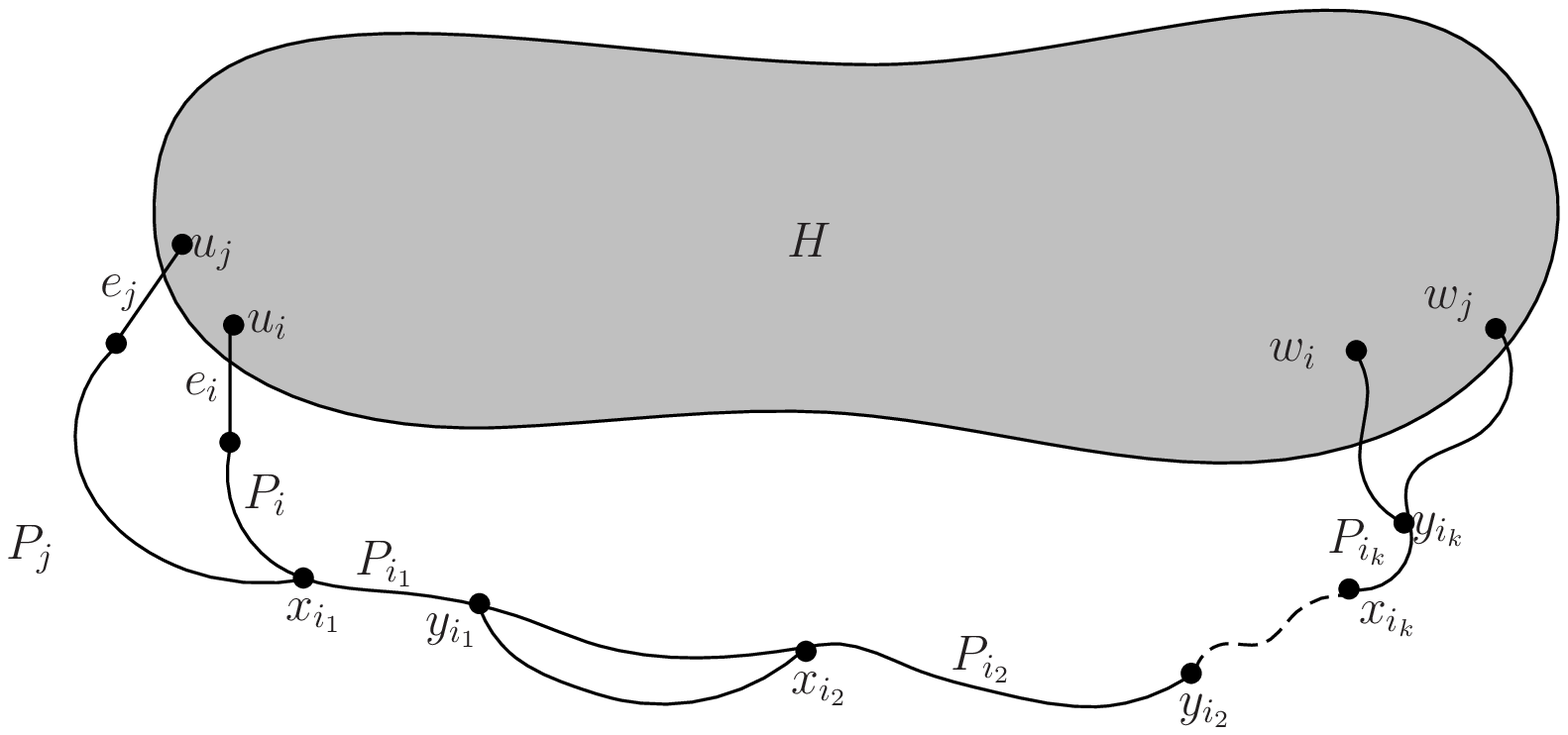}}\vspace{0.5cm}

Figure 1. Two $H$-ears $P_i$ and $P_j$.
\end{center}
\end{figure}
When $m\geq 2$, we first construct an optimal $(H,e_j)$-ear $P^*_j$
such that $P_i$ and $P^*_j$ has only one continuous common segment.
Let $P_{i_1},P_{i_2},\ldots,P_{i_m}$ be the $m$ continuous common
segments of $P_i$ and $P_j$ and they appear in $P_i$ in that order.
See Figure~1 for details. Furthermore, suppose that $x_{i_k}$ and
$y_{i_k}$ are the two ends of the path $P_{i_k}$ and they appear in
$P_i$ successively. We say that the following claim holds.

\noindent{\itshape Claim 1:
$\ell(y_kP_ix_{k+1})=\ell(y_kP_jx_{k+1})$ for any $1\leq k\leq
m-1$.}

If not, that is, there exists an integer $k$ such that
$\ell(y_kP_ix_{k+1})\neq\ell(y_kP_jx_{k+1})$. Without loss of
generality, we assume $\ell(y_kP_ix_{k+1})<\ell(y_kP_jx_{k+1})$.
Then we shall get a more shorter path $H$-ear containing $e_j$ by
replacing $y_kP_jx_{k+1}$ with $y_kP_ix_{k+1}$, a contradiction.
Thus $\ell(y_kP_ix_{k+1})=\ell(y_kP_jx_{k+1})$ for any $k$.

Let $P^*_j$ be the path obtained from $P_j$ by replacing
$y_kP_jx_{k+1}$ with $y_kP_ix_{k+1}$, and let $P_j=P^*_j$. If the
continuous common segment of $P_i$ and $P_j$ does not contain $w_j$.
Suppose $x$ and $y$ are the two ends of the common segment such that
$x$ and $y$ appeared on $P$ starting from $u_i$ to $w_i$
successively. Similar to Claim~1, $\ell(yP_iw_i)=\ell(yP_jw_j)$. Let
$P^*_j$ be the path obtained from $P_j$ by replacing $yP_jw_j$ with
$yP_iw_i$. Clearly, $P^*_j$ is our desired optimal $(H,u_jv_j)$-ear.
\end{proof}

\begin{lemma}\label{L2}
For every bridgeless graph $G$, $\eta(G)\leq\zeta(G)$.
\end{lemma}
\begin{proof}
Suppose that there exists an edge $e$ such that the length $\ell(C)$
of the smallest cycle $C$ containing $e$ is larger than $\zeta(G)$.
Then, $C$ is not an isometric cycle since the length of a largest
isometric cycle is $\zeta(G)$. Thus there exist two vertices $u$ and
$v$ on $C$ such that $d_{G}(u,v)<d_{C}(u,v)$. Let $P$ be a shortest
path between $u$ and $v$ in $G$. Then a closed trial $C'$ containing
$e$ is obtained from the segment of $C$ containing $e$ between $u$
and $v$ by adding $P$. Clearly, the length $\ell(C')$ is less than
$\ell(C)$. We can get a cycle $C''$ containing $e$ from $C'$. Thus
there exists a cycle $C''$ containing $e$ with length less than
$\ell(C)$, a contradiction. Therefore $\eta(G)\leq\zeta(G)$.
\end{proof}
\begin{lemma}\label{L3}
Let $G$ be a bridgeless graph and $u$ be a center of $G$. For any
$i\leq rad(G)-1$ and every leg $e$ of $N_i(u)$, there exists an
optimal $(N_i[u],e)$-ear with length at most
$\min\{2(rad(G)-i)+1,\eta(G)\}$.
\end{lemma}
\begin{proof}
Let $P$ be an optimal $(N_i[u],e)$-ear. Since $e$ belongs a cycle
with length at most $\eta(G)$, $\ell(P)\leq \eta(G)$. On the other
hand, if $\ell(P)\geq 2(rad(G)-i)+1$, then the middle vertex of $P$
has length at least $rad(G)-i+1$ from $N_i[u]$, a contradiction.
\end{proof}

\section{Oriented diameter}

At first, we have the following observation.
\begin{observation}
Let $G$ be a graph and $H$ be a bridgeless spanning subgraph of $G$.
Then the oriented radius (resp. diameter) of $G$ is not larger than
the oriented radius (resp. diameter) of $H$.
\end{observation}

\noindent{\itshape Proof of Theorem~\ref{T2}:} We only need to show
that $G$ has an orientation $H$ such that $rad(H)\leq
\sum_{i=1}^{rad(G)}\min\{2i,\eta(G)-1\}\leq rad(G)(\eta(G)-1)$. Let
$u$ be a center of $G$ and let $H_0$ be the trivial graph with
vertex set $\{u\}$. We assert that {\itshape there exists a subgraph
$G_i$ of $G$ such that $N_i[u]\subseteq V(G_i)$ and $G_i$ has an
orientation $H_i$ satisfying that $rad(H_i)\leq ecc_{H_i}(u)\leq
\Sigma_{j=1}^i\min\{2(rad(G)-j),\eta(G)-1\}$.}

{\itshape Basic step:} When $i=1$, we omit it since the proof of
this step is similar to that of the following induction step.

{\itshape Induction step:} Assume that the above assertion holds for
$i-1$. Next we show that the above assertion also holds for $i$. For
any $v\in N_i(u)$, either $v\in V(H_{i-1})$ or $v\in N(H_i)$ since
$N_{i-1}[u]\subseteq V(H_{i-1})$. If $N_i(u)\subseteq V(H_{i-1})$,
then let $H_i=H_{i-1}$ and we are done. Thus, we suppose
$N_i(u)\not\subseteq V(H_{i-1})$ in the following.

Let $X=N_i(u)\setminus V(H_{i-1})$. Pick $x_1\in X$, let $y_1$ be a
neighbor of $x_1$ in $H_{i-1}$ and let $P_1=P_{y_1z_1}$ be an
optimal $(H_{i-1},x_1y_1)$-ear. We orient $P$ such that $P_1$ is a
directed path. Pick $x_2\in X$ satisfying that all incident edges of
$x_2$ are not oriented. Let $y_2$ be a neighbor of $x_2$ in
$H_{i-1}$. If there exists an optimal $(H_{i-1},x_2y_2)$-ear $P_2$
such that $P_1$ and $P_2$ are independent, then we can orient $P_2$
such that $P_2$ is a directed path. Otherwise, by Lemma~\ref{L1}
there exists an optimal $(H_{i-1},x_2y_2)$-ear $P_2=P_{y_2z_2}$ such
that $P_1$ and $P_2$ has only one continuous common segment
containing $z_2$. Clearly, we can orient the edges in
$E(P_2)\setminus E(P_1)$ such that $P_2$ is a directed path. We can
pick the vertices of $X$ and oriented optimal $H$-ears similar to
the above method until that for any $x\in X$, at least two incident
edges of $x$ are oriented. Let $H_i$ be the graph obtained from
$H_{i-1}$ by adding vertices in $V(G)\setminus V(H_{i-1})$, which
has at least two new oriented incident edges, and adding new
oriented edges. Clearly, $N_i[u]\subseteq V(H_i)=V(G_i)$.

Now we show that $rad(H_i)\leq
\Sigma_{j=1}^i\min\{2(rad(G)-i),\eta(G)-1\}$. It suffices to show
that for every vertex $x$ of $H_i$, $d_{H_i}(H_{i-1},x)\leq
\min\{2(rad(G)-i),\eta(G)-1\}$ and $d_{H_i}(x,H_{i-1})\leq
\min\{2(rad(G)-i),\eta(G)-1\}$. If $x\in V(H_{i-1})$, then the
assertion holds by inductive hypothesis. If $x\not\in V(H_{i-1})$.
Let $P$ be a directed optimal $(H_i,e)$-ear containing $x$, where
$e$ is some leg of $H_{i-1}$ (such a leg and such an ear exists by
the definition of $H_i$. By Lemma~\ref{L3}, $\ell(P)\leq
\min\{2(rad(G)-i)+1,\eta(G)\}$. Thus, $d_{H_i}(x,H_{i-1})\leq
\min\{2(rad(G)-i),\eta(G)-1\}$ and $d_{H_i}(H_{i-1},x)\leq
\min\{2(rad(G)-i),\eta(G)-1\}$. Therefore, $rad(H_i)\leq
\Sigma_{j=1}^i\min\{2(rad(G)-j),\eta(G)-1\}$. \hfill$\Box$

\begin{remark}
The above theorem is optimal since it implies Chv\'atal and
Thomassen's optimal Theorem \ref{T1}. Readers can see \cite{chv,
kwo} for optimal examples.
\end{remark}

The following example shows that our result is better than that of
Theorem~\ref{T1}.

\begin{example}\upshape
Let $H_3$ be a triangle with one of its vertices designated as root.
In order to construct $H_r$, take two copies of $H_{r-1}$. Let $H_r$
be the graph obtained from the triangle $u_0,u_1,u_2$ by identifying
the root of first (resp. second) copy of $H_{r-1}$ with $u_1$ (resp.
$u_2$), and $u_0$ be the root of $H_r$. Let $G_r$ be the graph
obtained by taking two copies of $H_r$ and identifying their roots.
See Figure~2 for details. It is easy to check that $G_r$ has radius
$r$ and every edge belongs to a cycle of length $\eta(G)=3$. By
Theorem~\ref{T1}, $G_r$ has an orientation $H_r$ such that
$rad(H_r)\leq r^2+r$ and $diam(H_r)\leq 2r^2+2r$. But, by
Theorem~\ref{T2}, $G_r$ has an orientation $H_r$ such that
$rad(G)\leq 2r$ and $diam(G)\leq 4r$. On the other hand, it is easy
to check that all the strong orientations of $G_r$ has radius $2r$
and diameter $4r$.
\end{example}
\begin{figure}[h,t,b]
\begin{center}
\scalebox{0.8}[0.8]{\includegraphics{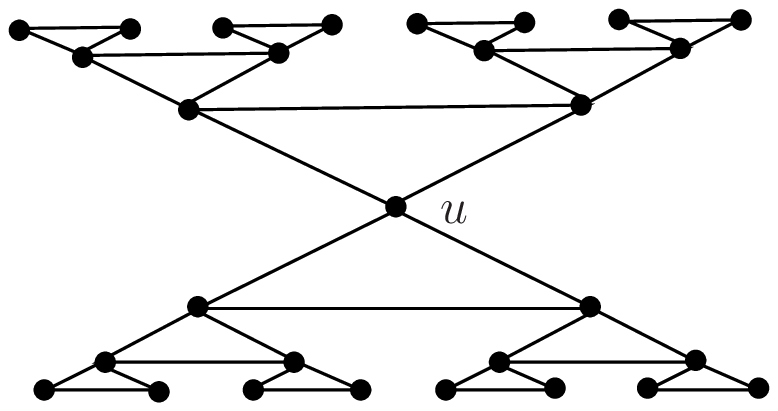}}

\vspace{0.5cm} Figure~2. The graph $G_3$ which has oriented\\ radius
$6$
 and oriented diameter $12$.
\end{center}
\end{figure}

We have the following result for plane graphs.

\begin{theorem}
Let $G$ be a plane graph. If the length of the boundary of every
face is at most $k$, then $G$ has an oriented $H$ such that
$rad(H)\leq rad(G)(k-1)$ and $diam(H)\leq 2rad(G)(k-1)$.
\end{theorem}
Since every edge of a maximal plane (resp. outerplane) graph belongs
to a cycle with length $3$, the following corollary holds.
\begin{corollary}
Let $G$ be a maximal plane (resp. outerplane) graph. Then there
exists an orientation $H$ of $G$ such that $rad(H)\leq 2rad(G)$ and
$rad(H)\leq 4rad(G)$.
\end{corollary}

A graph $G$ is $edge$-$transitive$ if for any $e_1,e_2\in E(G)$,
there exists an automorphism $g$ such that $g(e_1)=e_2$. We have the
following result for edge-transitive graphs.

\begin{theorem}
Let $G$ be a bridgeless edge-transitive graph. Then $G$ has an
orientation $H$ such that $rad(H)\leq rad(G)(g(G)-1)$ and
$diam(H)\leq 2rad(G)(g(G)-1)$, where $g(G)$ is the girth of $G$,
that is, the length of a smallest induced cycle.
\end{theorem}

For general bipartite graphs, the following theorem holds.
\begin{theorem}\label{T5}
Let $G=(V_1\cup V_2,E)$ be a bipartite graph with $|V_1|=n$ and
$|V_2|=m$. If $d(x)\geq k>\lceil m/2 \rceil$ for any $x\in V_1$,
$d(y)\geq r>\lceil n/2 \rceil$ for any $y\in V_2$, then there exists
an orientation $H$ of $G$ such that $rad(H)\leq 9$.
\end{theorem}

\begin{proof}
It suffices to show that $rad(G)\leq 3$ and $\eta(G)\leq 4$ by
Theorem~\ref{T2}.

First, we show that $rad(G)\leq 3$. Fix a vertex $x$ in $G$, and let
$y$ be any vertex different from $x$ in $G$. If $x$ and $y$ belong
to the same part, without loss of generality, say $x,y\in V_1$. Let
$X$ and $Y$ be neighborhoods of $x$ and $y$ in $V_2$, respectively.
If $X\cap Y=\emptyset$, then $|V_2|\geq |X|+|Y|\geq 2k>m$, a
contradiction. Thus $X\cap Y\neq \emptyset$, that is, there exists a
path between $x$ and $y$ of length two. If $x$ and $y$ belong to
different parts, without loss of generality, say $x\in V_1,y\in
V_2$. Suppose $x$ and $y$ are nonadjacent, otherwise there is
nothing to do. Let $X$ and $Y$ be neighborhoods of $x$ and $y$ in
$G$, and let $X'$ be the set of neighbors except for $x$ of $X$ in
$G$. If $X'\cap Y=\emptyset$, then $|V_1|\geq 1+|Y|+|X'|\geq
1+r+(r-1)=2r>n$, a contradiction (Note that $|X'|\geq r-1$). Thus
$X'\cap Y\neq\emptyset$, that is, there exists a path between $x$
and $y$ of length three in $G$.

Next we show that $\eta(G)\leq 4$. Let $xy$ be any edge in $G$. Let
$X$ be the set of neighbors of $x$ except for $y$ in $G$, let $Y$ be
the set of neighbors of $y$ except for $x$ in $G$, let $X'$ be the
set of neighbors except for $x$ of $X$ in $G$. If $X'\cap
Y=\emptyset$, then $|V_1|\geq 1+|Y|+|X'|\geq 1+(r-1)+(r-1)=2r-1>n$,
a contradiction (Note that $|X'|\geq r-1$). Thus $X'\cap
Y\neq\emptyset$, that is, there exists a cycle containing $xy$ of
length four in $G$.
\end{proof}

\begin{remark}
The degree condition is optimal. Let $m, n$ be two even numbers with
$n,m\geq 2$. Since $K_{n/2,m/2}\cup K_{n/2,m/2}$ is disconnected,
the oriented radius (resp. diameter) of $K_{n/2,m/2}\cup
K_{n/2,m/2}$ is $\infty$.
\end{remark}

For equal bipartition $k$-regular graph, the following corollary
holds.

\begin{corollary}
Let $G=(V_1\cup V_2, E)$ be a $k$-regular bipartite graph with
$|V_1|=|V_2|=n$. If $k> n/2$, then there exists an orientation $H$
of $G$ such that $rad(H)\leq 9$.
\end{corollary}

The following theorem holds for general graphs.

\begin{theorem}
Let $G$ be a graph.

$(i)$ If there exists an integer $k\geq 2$ such that
$|N_k(u)|>n/2-1$ for every vertex $u$ in $G$, then $G$ has an
orientation $H$ such that $rad(H)\leq 4k^2$ and $diam(H)\leq 8k^2$.

$(ii)$ If $\delta(G)>n/2$, then $G$ has an orientation $H$ such that
$rad(H)\leq 4$ and $diam(H)\leq 8$.
\end{theorem}
\begin{proof}
Since methods of proofs of $(i)$ and $(ii)$ are similar, we only
prove $(i)$. For $(i)$, it suffices to show that $rad(G)\leq 2k$ and
$\eta(G)\leq 2k+1$ by Theorem~\ref{T2}.

We first show $rad(G)\leq 2k$. Fix $u$ in $G$, for every $v\in
V(G)$, if $v\in N_k[u]$, then $d(u,v)\leq k$. Suppose $v\not\in
N_k[u]$, we have $N_k(u)\cap N_k(v)\neq\emptyset$. If not, that is,
$N_k(u)\cap N_k(v)=\emptyset$, then $|N_k(u)|+|N_k(v)|+2>n$ (a
contradiction). Thus $d(u,v)\leq 2k$.

Next we show $\eta(G)\leq 2k+1$. Let $e=uv$ be any edge in $G$. If
$N_k(u)\cap N_k(v)=\emptyset$, then $|V(G)|\geq
|N_k(u)|+|N_k(v)|+2>n$, a contradiction. Thus $N_k(u)\cap
N_k(v)\neq\emptyset$. Pick $w\in N_k(u)\cap N_k(v)$, and let $P$
(resp. $Q$) be a path between $u$ and $w$ (resp. between $v$ and
$w$). Then $e$ belongs a close trial $uPwQvu$ of length $2k+1$.
Therefore, $e$ belongs a cycle of length at most $2k+1$.
\end{proof}

\begin{remark} The above condition is almost optimal since $K_{n/2}\cup
K_{n/2}$ is disconnected for even $n$.
\end{remark}

\begin{corollary}
Let $G$ be a graph with minimum degree $\delta(G)$ and girth $g(G)$.
If there exists an integer $k$ such that $k\leq g(G)/2$ and
$\delta(G)(\delta(G)-1)^{k-1}>n/2-1$, then $G$ has an orientation
$H$ such that $rad(H)\leq 4k^2$.
\end{corollary}
\begin{proof}
Let $k$ be an integer such that  $k\leq g(G)/2$ and
$\delta(G)(\delta(G)-1)^{k-1}>n/2-1$. For any vertex $u$ of $G$, let
$1\leq i<k$ be any integer and $x,y\in N_i(u)$. If $x$ and $y$ have
a common neighbor $z$ in $N_{i+1}(u)$, then $G$ has a cycle of
length at most $2i<2k\leq g(G)/2$, a contradiction. Thus $x$ and $y$
has no common neighbor in $N_{i+1}(u)$. Therefore, $|N_k(u)|\geq
\delta(G)(\delta(G)-1)^{k-1}>n/2-1$. By Theorem~\ref{T2}, $G$ has an
orientation $H$ such that $rad(H)\leq 4k^2$.
\end{proof}

\section{Upper bound for rainbow connection number}

At first, we have the following observation.
\begin{observation}
Let $G$ be a graph and $H$ be a spanning subgraph of $G$. Then
$rc(H)\leq rc(G)$.
\end{observation}

\noindent{\itshape Proof of Theorem~\ref{T4}:} Let $u$ be a center
of $G$ and let $H_0$ be the trivial graph with vertex set $\{u\}$.
We assert that {\itshape there exists a subgraph $H_i$ of $G$ such
that $N_i[u]\subseteq V(H_i)$ and $rc(H_i)\leq \Sigma_{j=1}^i
\min\{2(rad(G)-j)+1,\eta(G)\}$.}

{\itshape Basic step:} When $i=1$, we omit it since the proof of
this step is similar to that of the following induction step.

{\itshape Induction step:} Assume that the above assertion holds for
$i-1$ and $c$ is a\linebreak[3] $rc(H_{i-1})$-rainbow coloring of
$H_{i-1}$. Next we show that the above assertion holds for $i$. For
any $v\in N_i(u)$, either $v\in V(H_{i-1})$ or $v\in N(H_i)$ since
$N_{i-1}[u]\subseteq V(H_{i-1})$. If $N_i(u)\subseteq V(H_{i-1})$,
then let $H_i=H_{i-1}$ and we are done. Thus, we suppose
$N_i(u)\not\subseteq V(H_{i-1})$ in the following.

Let $C_1=\{\alpha_1,\alpha_2,\cdots\}$ and
$C_2=\{\beta_1,\beta_2,\cdots\}$ be two pools of colors, none of
which are used to color $H_{i-1}$. An edge-coloring of an $H$-ear
$P=(u_0,u_1,\cdots,u_k)$ is a $symmetrical\ coloring$ if its edges
are colored by $\alpha_1,\alpha_2,\cdots,\alpha_{\lceil
k/2\rceil},\beta_{\lfloor k/2\rfloor},\cdots,\beta_2,\beta_1$ in
that order or $\beta_1,\beta_2,\cdots,\beta_{\lfloor
k/2\rfloor},\alpha_{\lceil k/2\rceil}\cdots,\alpha_2,\alpha_1$ in
that order.

Let $X=N_i(u)\setminus V(H_{i-1})$ and
$m=\min\{2(rad(G)-i)+1,\eta(G)\}$. Pick $x_1\in X$, Let $y_1$ be a
neighbor of $x_1$ in $H_{i-1}$ and $P_1$ be an optimal
$(H_{i-1},x_1y_1)$-ear. We can color $P$ symmetrically with colors
$\alpha_1,\alpha_2,\cdots,\alpha_{\lceil
\ell(P)/2\rceil},\beta_{\lfloor
\ell(P)/2\rfloor},\ldots,\beta_2,\beta_1$. Pick $x_2\in X$
satisfying that all the incident edges of $x_2$ are not colored. Let
$y_2$ be a neighbor of $x_2$ in $H_{i-1}$. If there exists an
optimal $(H_{i-1},x_2y_2)$-ear $P_2$ such that $P_1$ and $P_2$ are
independent, then we can color $P_2$ symmetrically with colors
$\alpha_1,\alpha_2,\cdots,\alpha_{\lceil
\ell(P_2)/2\rceil},\beta_{\lfloor
\ell(P_2)/2\rfloor},\ldots,\beta_2,\beta_1$. Otherwise, by
Lemma~\ref{L1}, there exists an optimal $(H_{i-1},x_2y_2)$-ear
$P_2=P_{y_2z_2}$ such that $P_1$ and $P_2$ have only one continuous
common segment containing $z_2$, where $z_2$ is the other foot of
$P_2$. Thus we can color $P_2$ symmetrically with colors
$\alpha_1,\alpha_2,\cdots,\alpha_{\lceil
\ell(P_2)/2\rceil},\beta_{\lfloor
\ell(P_2)/2\rfloor},\ldots,\beta_2,\beta_1$ by preserving the
coloring of $P_1$. We can pick the vertices of $X$ and color optimal
$H_i$-ears until that for any $x\in X$, at least two incident edges
of $x$ are colored. Since for any leg $e$ of $H_{i-1}$, $\ell(e)\leq
m$ by Lemma~\ref{L3}, we use at most $m$ coloring in the above
coloring process.

Let $H_i$ be the graph obtained from $H_{i-1}$ by adding vertices in
$V(G)\setminus V(H_{i-1})$, which has at least two new colored
incident edges, and adding new colored edges. Clearly,
$N_i[u]\subseteq V(H_i)$. It is suffices to show that $H_i$ is
rainbow connected. Let $x$ and $y$ be two distinct vertices in
$H_i$. If $x,y\in V(H_{i-1})$, then there exists a rainbow path
between $x$ and $y$ by inductive hypothesis. If exactly one of $x$
and $y$ belongs to $V(H_{i-1})$, say $x$.  Let $P$ be a symmetrical
colored $H_{i-1}$-ear containing $y$ and $y'$ be a foot of $P$.
There exists a rainbow path $Q$ between $x$ and $y'$ in $H_{i-1}$ by
inductive hypothesis. Thus, $xQy'Py$ is a rainbow path between $x$
and $y$ in $H_i$.

Suppose none of $x$ and $y$ belongs to $H_{i-1}$. Let $P$ and $Q$ be
symmetrical colored $H_{i-1}$-ear containing $x$ and $y$,
respectively. Furthermore, let $x',x''$ be the foot of $P$ and
$y',y''$ be the foot of $Q$. Without loss of generality, assume that
$P$ is colored from $x'$ to $x''$ by
$\alpha_1,\alpha_2,\cdots,\alpha_{\lceil\ell(P)/2\rceil},\beta_{\lfloor
\ell(P)/2\rfloor},\ldots,\beta_2,\beta_1$ in that order, and $Q$ is
colored from $y'$ to $y''$ by $\alpha_1,\alpha_2,
\cdots,\alpha_{\lceil \ell(Q)/2\rceil},\beta_{\lfloor
\ell(Q)/2\rfloor},\ldots,\beta_2,\beta_1$ in that order. If
$\ell(x'Px)\leq \ell(y'Qy)$. Let $R$ be a rainbow path between $x'$
and $y''$ in $H_{i-1}$. Then $xPx'Ry''Qy$ is a rainbow path between
$x$ and $y$ in $H_i$. Otherwise, $\ell(x'Px)> \ell(y'Qy)$. Let $R$
be a rainbow path between $y'$ and $x''$ in $H_{i-1}$. Then
$yPy'Rx''Qx$ is a rainbow path between $x$ and $y$ in $H_i$. Thus,
there exists a rainbow path between any two distinct vertices in
$H_i$, that is, $H_i$ is $(\Sigma_{j=1}^i
\min\{2(rad(G)-j)+1,\eta(G)\})$-rainbow connected. \hfill$\Box$

The following optimal example is from \cite{bas}.

\begin{figure}[h,t,b]
\begin{center}
\scalebox{0.8}[0.8]{\includegraphics{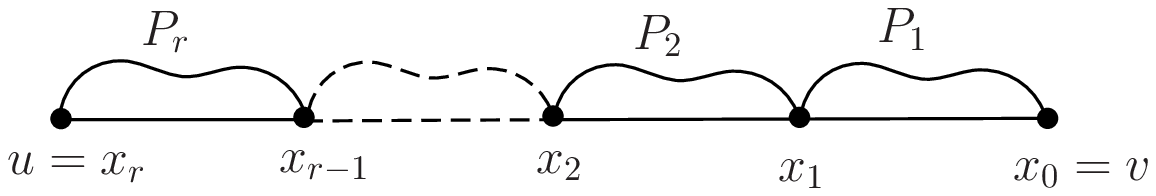}} \vspace{0.5cm}

Figure~3. Graph $H_{r,\eta(G)}$. Every $P_i$ is a path between $x_i$\\
and $x_{i-1}$ of length $\ell(P_i)=\min\{2i, \eta(G)-1\}$.
\end{center}
\end{figure}

\begin{example}\upshape
For any $r\geq 1$ and $3\leq\eta(G)\leq 2r + 1$, we first construct
the graph $H_{r,\eta(G)}$ as in Figure~3. Clearly, $H_{r,\eta(G)}$
is a bridgeless graph with radius $rad(G)=ecc(u)=r$ and every edge
of $H_{r,\eta(G)}$ is contained in a cycle of length at most
$\eta(G)$.

Let $m=\sum_{i=1}^r\min\{2i+1,\eta(G)\}$ and let $H^j$ be a copy of
$H_{r,\eta(G)}$, where $1\leq j\leq m^r+1$, and $V(H^j)=\{x^j : x\in
V(H_{r,\eta(G)})$ and $E(H_j)=\{x^jy^j\,|\, xy\in
E(H_{r,\eta(G)})\}$. Identify the vertex $u^j$ as a new vertex $u$.
The resulting graph is denoted by $G$. It is easy to check that $G$
is a bridgeless graph with radius $rad(G)=ecc(u)=r$ and every edge
of $H_{r,\eta(G)}$ is contained in a cycle of length at most
$\eta(G)$. Thus, $rc(G)\leq\Sigma_{i=1}^r \min\{2i+1, \eta(G)\}$ by
Theorem~\ref{T4}. On the other hand, for any $k< m$ and any $k$-edge
coloring of $G$, every $r$-length $P_{uv^j}$ path can be colored in
at most $k^r$ different ways. By the Pigeonhole Principle, there
exist $p\neq q,\,1\leq p<q\leq m^r+1$ such that
$c(x^p_{i-1}x^p_i)=c(x^q_{i-1}x^q_i)$ for $1\leq i\leq r$. Consider
any rainbow path $P$ from $v^p$ to $v^q$. For every $1\leq i\leq r$,
$x^p_{i-1}x^p_i$ belongs to $P$ if and only if $x^q_{i-1}x^q_i$ does
not belong to $P$. Thus, $\ell(P)\geq\Sigma_{i=1}^r \min\{2i+1,
\eta(G)\}=m$, and there does not exist any rainbow path between
$v^p$ and $v^q$. Hence, $rc(G)=\sum_{i=1}^r\min\{2i+1,\eta(G)\}$.
\end{example}

The following example shows that our result is better than that of
Theorem~\ref{T3}.
\begin{example}\upshape
Let $r\geq 3,k\geq 2r$ be two integers, and $W_k=C_k\vee K_1$ be an
wheel, where $V(C_k)=\{u_1,u_2,\ldots,u_k\}$ and $V(K_1)=\{u\}$. Let
$H$ be the graph obtained from $W_k$ by inserting $r-1$ vertices
between every edge $uu_i,\, 1\leq i\leq k$. For every edge $e=xy$ of
$H$, add a new vertex $v_e$ and new edges $v_ex,v_ey$. Denote by $G$
the resulting graph. It is easy to check that $rad(G)=r,\,
diam(G)=2r$, $\eta(G)=3$ and $\zeta(G)=2r-1$. By Theorem~\ref{T2},
we have $rc(G)\leq\sum_{i=1}^r\min\{2i+1,\zeta(G)\}\leq r^2+2r-2$.
But, by Theorem~\ref{T5} we have $rc(G)\leq 3r$. On the other hand,
$rc(G)\geq 2r$ since $diam(G)=2r$.
\end{example}
The remaining results are similar to those in Section~3.
\begin{theorem}
Let $G$ be a plane graph. If the length of the boundary of every
face is at most $k$, then $rc(G)\leq k\,rad(G)$.
\end{theorem}
\begin{corollary}
Let $G$ be a maximal plane (resp. outerplane) graph. Then $rc(G)\leq
3rad(G)$.
\end{corollary}
\begin{theorem}
Let $G$ be a bridgeless edge-transitive graph. Then $rc(G)\leq
rad(G)g(G)$, where $g(G)$ is the girth of $G$.
\end{theorem}
\begin{theorem}
Let $G=(V_1\cup V_2,E)$ be a bipartite graph with $|V_1|=n$ and
$|V_2|=m$. If $d(x)\geq k>\lceil m/2\rceil$ for any $x\in V_1$,
$d(y)\geq r>\lceil n/2\rceil$ for any $y\in V_2$, then $rc(G)\leq
12$.
\end{theorem}

\begin{remark}
The degree condition is optimal. Let $m, n$ be two even numbers with
$n, m\geq 2$. Since $K_{n/2,m/2}\cup K_{n/2,m/2}$ is disconnected,
$rc(K_{n/2,m/2}\cup K_{n/2,m/2})=\infty$.
\end{remark}

\begin{corollary}
Let $G=(V_1\cup V_2,E)$ be a $k$-regular bipartite graph with
$|V_1|=|V_2|=n$. If $k>\lceil n/2\rceil$, then $rc(G)\leq 12$.
\end{corollary}

The following theorem holds for general graphs.

\begin{theorem}
Let $G$ be a graph.

$(i)$ If there exists an integer $k\geq 2$ such that
$|N_k(u)|>n/2-1$ for every vertex $u$ in $G$, then $rc(G)\leq
4k^2+2k$.

$(ii)$ If $\delta(G)>n/2$, then $rc(G)\leq 6$.
\end{theorem}

\begin{remark}
The above condition is almost optimal since $K_{n/2}\cup K_{n/2}$ is
disconnected for even $n$.
\end{remark}

\begin{corollary}
Let $G$ be a graph with minimum degree $\delta(G)$ and girth $g(G)$.
If there exists an integer $k$ such that $k<g(G)/2$ and
$\delta(G)(\delta(G)-1)^{k-1}>n/2-1$, then then $rc(G)\leq 4k^2+2k$.
\end{corollary}

\end{document}